\documentclass[11pt]{article}
\usepackage[T1]{fontenc}
\usepackage{lmodern}
\usepackage[margin=1in]{geometry}
\usepackage{amsmath,amssymb,amsthm,mathtools}
\usepackage{booktabs,array,longtable}
\usepackage{microtype}
\usepackage{xcolor}
\usepackage{xurl}
\usepackage{enumitem}
\usepackage[colorlinks=true,linkcolor=blue!50!black,citecolor=blue!50!black,urlcolor=blue!50!black]{hyperref}
\hypersetup{pdftitle={PL standardness of the four-pentachoron census sphere Q},pdfsubject={A reproducible certificate of PL standardness},pdfkeywords={four-manifold, triangulation, Cappell--Shaneson sphere, PL topology, computer-assisted proof}}
\setlist{itemsep=2pt,topsep=5pt}
\newtheorem{theorem}{Theorem}[section]
\newtheorem{proposition}[theorem]{Proposition}
\newtheorem{lemma}[theorem]{Lemma}
\newtheorem{corollary}[theorem]{Corollary}
\theoremstyle{definition}

\theoremstyle{remark}

\newcommand{\PL}{\mathrm{PL}}
\newcommand{\congpl}{\cong_{\PL}}
\newcommand{\Z}{\mathbb Z}

\newcommand{\Int}{\operatorname{int}}
\newcommand{\cl}{\operatorname{cl}}
\newcommand{\core}{\texttt{core42}}
\newcommand{\sig}[1]{\texttt{\detokenize{#1}}}
\newcommand{\file}[1]{\path{#1}}

\title{PL standardness of the four-pentachoron census sphere \(Q\)}
\author{Vance Faber\thanks{Hoquiam, WA \texttt{vance.faber@gmail.com}} \and Michael Murphy\thanks{Venice, CA}}
\date{19 October 2026}

\begin{document}
\maketitle

\begin{abstract}
We prove that the four-pentachoron homotopy \(4\)-sphere \(Q\) singled out
in the census of Burke, Burton and Spreer is PL homeomorphic to the
standard \(4\)-sphere. The key construction identifies the compact
exterior of an explicit locally flat \(2\)-sphere in an auxiliary
39-vertex model of \(Q\) with the exterior of Budney, Burton and
Hillman, then certifies its fibration by a normal cut followed by a
relative collapse. The exterior identification uses finite checked
moves, an exact regular-neighbourhood construction and a PL lemma
controlling the passage through an ideal triangulation. Cutting the
compact exterior along the specified normal hypersurface gives a
manifold that collapses onto one boundary copy of the punctured
\(3\)-torus in 209,144 elementary pairs; a marked product lemma
identifies the cut as a product. The resulting integral monodromy
realizes \(Q\) as a Cappell--Shaneson filling whose two surgery
framings are standard by Gompf's results. The fibration certificate
is independent of Issa's earlier PL-compatible constructions.
Combined with the published census, the result establishes
standardness for every homotopy \(4\)-sphere triangulated with at most
four pentachora.
\end{abstract}

\section{The census sphere and the main result}\label{sec:theorem}

Burke, Burton and Spreer identify a unique four-pentachoron triangulation
of a homotopy \(4\)-sphere which their census computations did not
connect to the standard sphere~\cite[Section~5]{BBS2026}. We use their
notation \(Q\) for this triangulation. It is specified by the Regina
isomorphism signature
\begin{equation}\label{eq:Qsig}
  \sig{eAMPcaabcddd+aoa+aAa8aQara}.
\end{equation}
Its face vector is \((2,2,8,10,4)\). A pentachoron is a \(4\)-simplex;
a generalized triangulation permits identifications between faces of
the same simplex. Throughout, PL standardness means PL homeomorphism
to the sphere with its standard PL structure.

\begin{theorem}\label{thm:main}
The four-pentachoron triangulation \(Q\) in \eqref{eq:Qsig} is
PL homeomorphic to \(S^4\).
\end{theorem}

\begin{corollary}[The census through four pentachora]\label{cor:census}
Every closed orientable generalized triangulation of a homotopy
\(4\)-sphere with at most four pentachora is PL standard.
\end{corollary}

\begin{proof}
A closed triangulation has an even number of pentachora, since
\(5f_4=2f_3\). The published census identifies the two-pentachoron cases
and all but \(Q\) among the 647 four-pentachoron cases with the standard
PL sphere~\cite[Tables~1--2 and Section~5]{BBS2026}.
Theorem~\ref{thm:main} supplies the remaining identification.
\end{proof}

The corollary uses the published census classification. The
computations in this paper concern the remaining triangulation \(Q\);
they do not repeat that enumeration.

\paragraph{Proof strategy.}
Let \(M\) be the compact exterior obtained by truncating the ideal
vertex of the two-pentachoron model of Budney, Burton and Hillman
(BBH)~\cite{BBH2012}. Our main geometric identification is
\begin{equation}\label{eq:chain}
 \begin{gathered}
 Q\congpl K=N(S)\cup_{\partial N(S)}E,\\
 N(S)\congpl S^2\times D^2,\qquad E\congpl M.
 \end{gathered}
\end{equation}
Here \(K\) is a supplied simplicial triangulation with 39 vertices
and 832 pentachora, and \(S\subset K\) is an explicitly marked
\(2\)-sphere. A 137-move certificate identifies \(K\) with \(Q\).
The purpose of the larger model is to make the sphere, its links and
its regular neighbourhood accessible through simplex incidences.
Section~\ref{sec:sphere} constructs the decomposition, and
Section~\ref{sec:exterior} identifies the actual compact exterior.

An independent normal-cut and collapse certificate shows that \(M\)
is a PL bundle over \(S^1\) with fiber a punctured \(3\)-torus
\(T^3\setminus\Int D^3\) (Section~\ref{sec:fiber}).
Its integral monodromy identifies \(M\) with a Cappell--Shaneson
exterior for which both fillings by \(S^2\times D^2\) are standard
by Gompf's results (Section~\ref{sec:CS}). Thus \eqref{eq:chain}
proves the theorem without having to determine the surgery framing.
We reserve \(M^{\mathrm{id}}\) for the ideal presentation and \(M\)
for its compact truncation.

\paragraph{Relation to earlier work.}
The census authors already observed a Cappell--Shaneson connection
through a two-pentachoron subcomplex of \(Q\)~\cite[Section~5]{BBS2026}.
That subcomplex has singular links; an additional facet identification
produces the BBH ideal model. It is therefore not itself a compact
manifold exterior, and the additional identification does not give a
decomposition of \(Q\). The new identification here exhibits an
actual locally flat sphere in a triangulation of \(Q\), with regular
neighbourhood complement PL homeomorphic to the \emph{compact}
BBH exterior.
BBH identified their model topologically and posed the PL
identification in Problem~2.8~\cite{BBH2012}.
Issa's thesis supplies PL-compatible Cappell--Shaneson
triangulations~\cite{Issa2017}; Budney explicitly records that this
construction includes the BBH model~\cite{BudneyPublications}.
Our normal cut and relative collapse give an independent finite
certificate of that bundle structure.

The specific cut-open ball \(B_C^4\) in the census paper is a
separate object: Theorem~\ref{thm:main} does not settle its
Conjecture~2~\cite{BBS2026}. An explicit closed Pachner path from
\(Q\) to a standard sphere also remains to be constructed.

All computational assertions used below have fixed inputs and
replayable certificates. Section~\ref{sec:verification} distinguishes
those assertions from the PL theorems that interpret them, and
specifies the shared Regina kernel. Appendix~\ref{sec:hashes}
matches the mathematical notation to the data files.

\section{An auxiliary simplicial model and its marked sphere}
\label{sec:sphere}

\subsection{A certified identification with \texorpdfstring{\(Q\)}{Q}}
\label{sec:closed}

Let \(K\) be the simplicial complex specified by the facet file and
checksum in Appendix~\ref{sec:hashes}. Its face vector is
\[
 f(K)=(39,527,1738,2080,832).
\]
The certificate takes \(K\) to exactly \eqref{eq:Qsig} in 137
generalized triangulation moves: 37 edge collapses, 19 \(4\)--\(2\)
moves, 14 triangle \(2\)--\(0\) moves, 22 edge \(2\)--\(0\) moves
and 45 \(3\)--\(3\) moves. It specifies the initial simplex order,
each selected face and an explicit final relabelling. A separate
replayer rebuilds the facet complex, performs every move with
Regina's legality checks enabled and verifies the endpoint.
These checked moves establish \(K\congpl Q\); they include
generalized moves in addition to ordinary bistellar moves.

The endpoint is closed, connected, orientable and valid, with both
vertex links recognized by Regina as \(S^3\). Its fundamental group is trivial and
its first and second integral homology groups vanish. In particular,
\(H_2(K;\Z)=0\), which will determine the normal Euler number.
The identification of \(K\) with \(Q\) uses only this fixed facet
list and certificate.

\subsection{The embedded sphere}

Let \(u,v,a,b,c\) be the five vertices of \(K\) specified in
Appendix~\ref{sec:hashes}. Define \(S\) to be the suspension of
the edge cycle \(a b c a\), with suspension vertices \(u,v\).
Its six triangles are
\begin{align*}
 &(u,a,b),\quad (u,a,c),\quad (u,b,c),\\
 &(v,a,b),\quad (v,a,c),\quad (v,b,c).
\end{align*}
All six are present in \(K\), and together with their faces they
form a simplicial \(2\)-sphere with face vector \((5,9,6)\).
The ambient edge \(uv\) is not a face of \(S\).

\subsection{Local flatness and the normal neighbourhood}

For each vertex \(x\) of \(S\), the checker constructs the
pair \((\operatorname{lk}_{K}x,\operatorname{lk}_Sx)\). The ambient
link is recognized as \(S^3\). The cycle is full, and its constructed
compact knot exterior is recognized as a solid torus. Thus
the cycle is an unknot in the ambient \(S^3\). The two apical cycles have
three edges, and the three equatorial cycles have four. The five original
ambient links, in vertex order \((u,v,a,b,c)\), have
respectively 234, 225, 45, 41 and 86 tetrahedra. Each marked link
pair is constructed and checked independently. Coning a PL
homeomorphism to the standard pair \((S^3,S^1)\) and restricting
to the open cones identifies each open-star pair with
\((\Int D^4,\Int D^2)\). Indeed, for any subcomplex \(S\subset K\),
\[
 \operatorname{st}^{\circ}_K(x)\cap |S|
 =\operatorname{st}^{\circ}_S(x),
\]
because every point has a unique carrier simplex. This identity
does not require fullness. The five open stars cover \(|S|\),
proving local flatness at every point.

\begin{lemma}[The normal neighbourhood]\label{lem:normal}
Let \(S\subset\Int X\) be an oriented locally flat PL \(2\)-sphere
in an oriented PL \(4\)-manifold \(X\). If its normal Euler number
is zero, then its regular
neighbourhood, with its zero section, is PL homeomorphic to
\((S^2\times D^2,S^2\times\{0\})\).
\end{lemma}

\begin{proof}
The codimension-two normal-neighbourhood theorem identifies a regular
neighbourhood of a locally flat PL submanifold with its PL normal
\(2\)-disk bundle~\cite{Wall1967}. Only the existence part of Wall's
result is needed here. Oriented PL \(2\)-disk bundles over \(S^2\)
are classified by their Euler number: the comparison of PL
\(2\)-disk and \(2\)-plane bundles with real \(2\)-plane bundles
reduces the oriented clutching invariant to
\(\pi_1(SO(2))\cong\Z\)~\cite{Scott1971}; the full comparison proof
is in~\cite[Chapter~III, Theorem~2.1]{Scott1968}. Euler number zero
gives the trivial disk bundle. Regular-neighbourhood uniqueness
identifies this model with the chosen regular
neighbourhood~\cite[Chapter~3]{RS1972}.
\end{proof}

The normal Euler number is the self-intersection of \([S]\), computed by
an oriented normal push-off. Since \(H_2(K;\Z)=0\), it is zero.
Lemma~\ref{lem:normal} therefore gives \(N(S)\congpl S^2\times D^2\).

\subsection{A geometrically certified neighbourhood complement}

The vertex set of \(S\) spans one extra edge, \(uv\), in
\(K\). Subdivide that edge stellarly, fixing \(S\). Its degree is
48, so this replaces the 832 pentachora by 880 and makes \(S\) a full
subcomplex. On this subdivision define the simplexwise linear function
\(h\) to be zero at the vertices of \(S\) and one at every other vertex.
Set
\begin{equation}\label{eq:epsilon}
 N=\{h\leq\tfrac12\},\qquad
 E=\{h\geq\tfrac12\},\qquad
 J=\{h=\tfrac12\}.
\end{equation}
The full-subcomplex epsilon-neighbourhood construction is a regular
neighbourhood~\cite[Chapter~3]{RS1972}. Consequently \(N=N(S)\),
\(J=\partial N=\partial E\), and \(K\congpl N\cup_J E\).

A global pulling order triangulates the convex pieces in
\eqref{eq:epsilon}. All coordinates are rational. The construction uses
151 geometric vertices: 39 original vertices, the stellar vertex and
111 cut points. The resulting counts are
\begin{center}
\begin{tabular}{@{}lrl@{}}
\toprule
Piece & Top-dimensional simplices & Face vector\\
\midrule
\(E\) & 2517 & \((146,1933,5910,6640,2517)\)\\
\(N\) & 855 & \((116,1006,2522,2485,855)\)\\
\(J\) & 695 & \((111,806,1390,695)\)\\
\bottomrule
\end{tabular}
\end{center}
An independent geometric checker reconstructs 1383 polytopal face
posets, checks all 3372 nonzero rational simplex determinants and the
exact volumes, verifies oriented boundary cancellation and verifies the
entire interface in every dimension. This certifies a PL subdivision
of the marked pair. It does not infer the identity of the exterior from
homology or a group presentation.

\section{Identifying the BBH exterior}\label{sec:exterior}

\subsection{The compact and ideal move certificates}

The compact exterior undergoes 187 checked moves to a 216-pentachoron
triangulation with 70 real boundary tetrahedra. Cone each of these
tetrahedra to a single new vertex. The resulting 286-pentachoron
triangulation has one ideal vertex and no real boundary. A further
445 checked moves lead to
\begin{equation}\label{eq:Msig}
 M^{\mathrm{id}}=\sig{cMkabbb+aAa3blb},
\end{equation}
which is explicitly isomorphic to Regina's
\file{Example4.cappellShaneson()} and to the BBH ideal model. Thus the
exterior certificate has the form
\begin{equation}\label{eq:exteriorpath}
 E_{2517}\xrightarrow{187}E_{216}
 \xrightarrow{\text{70 boundary cones}}\widehat E_{286}
 \xrightarrow{445}M^{\mathrm{id}}_2.
\end{equation}
The subscripts record the numbers of pentachora.
The coning step is a separate construction. The number 632 counts the
187 compact and 445 ideal generalized moves, not the cones.

\begin{table}[htbp]
\centering
\caption{Moves in the exterior identification. The \(4\)--\(4\) operation
is Regina's size-preserving generalized move, replacing four pentachora
around an edge by four around a new edge.}
\label{tab:moves}
\begin{tabular}{@{}lrrr@{}}
\toprule
Operation & Compact & Ideal & Total\\
\midrule
Edge collapse & 145 & 3 & 148\\
\(4\)--\(2\) & 11 & 70 & 81\\
Triangle \(2\)--\(0\) & 17 & 49 & 66\\
Edge \(2\)--\(0\) & 7 & 49 & 56\\
Boundary shelling & 7 & 0 & 7\\
\(2\)--\(4\) & 0 & 27 & 27\\
\(1\)--\(5\) & 0 & 4 & 4\\
\(3\)--\(3\) & 0 & 223 & 223\\
\(4\)--\(4\) & 0 & 18 & 18\\
\(5\)--\(1\) & 0 & 2 & 2\\
\midrule
Total & 187 & 445 & 632\\
\bottomrule
\end{tabular}
\end{table}

The cone checker verifies every new simplex and gluing: 70 cones,
280 directed lateral gluings and 1400 vertex equations. The endpoint
isomorphism is checked on simplex and vertex maps, including all 50
gluing-vertex equations. The ideal path retains exactly one ideal
vertex. Here the finite coned polyhedron has Euler characteristic one;
its open manifold, or a compact truncation, has Euler characteristic
zero. These conventions must not be interchanged.

\subsection{Recovering the compact exterior}

The ideal moves must identify the compact exteriors, not only their
open interiors. We use the stronger data supplied by the certificate:
a finite PL equivalence of the pointed cone compactifications.

\begin{lemma}[Compact cores of pointed cone compactifications]
\label{lem:cone}
Let \(E_0,E_1\) be compact PL \(n\)-manifolds with nonempty connected
boundary. Write
\(\widehat E_i=E_i\cup_{\partial E_i}C(\partial E_i)\), with cone point
\(p_i\). A PL homeomorphism
\((\widehat E_0,p_0)\to(\widehat E_1,p_1)\) induces a PL homeomorphism
\(E_0\congpl E_1\) as manifolds with boundary. No boundary
parametrization is prescribed.
\end{lemma}

\begin{proof}
Choose a small closed apical subcone \(R_i\) in each attached cone.
After subdivision, \(R_i\) is a regular neighbourhood of \(p_i\) in
the polyhedron \(\widehat E_i\). If \(f\) is the pointed PL
homeomorphism, \(f(R_0)\) and \(R_1\) are regular neighbourhoods of
the same point in \(\widehat E_1\). The regular-neighbourhood
uniqueness theorem for polyhedra gives an ambient PL homeomorphism
\(a\), fixed at \(p_1\), carrying \(f(R_0)\) onto \(R_1\)
\cite[Theorem~3.8, p.~33]{RS1972}. This theorem allows a polyhedral
ambient space, including the nonmanifold cone point here.
Restrict \(af\) to the complements of
their interiors. Each complement is \(E_i\) with a finite outward
boundary collar attached. An inward collar of \(\partial E_i\) and
a PL reparametrization of its interval absorb that outward collar.
The restricted map therefore gives the required compact PL
homeomorphism.
\end{proof}

For boundaries \(S^2\times S^1\), the cone point is the unique
nonmanifold point: its local integral homology in degree two is
\(\widetilde H_1(S^2\times S^1;\Z)=\Z\), whereas a \(4\)-manifold
point has no such local homology. Thus every PL homeomorphism of
these compactifications is automatically pointed.

\subsection{The finite--ideal edge collapse}

The ideal path has three edge collapses. At steps 204 and 337 both
endpoints are finite. At step 339, edge 45 has degree two and joins
a finite vertex to the unique ideal vertex. The end-type audit
independently recognizes all finite endpoint links as \(S^3\) and
all three edge links as \(S^2\); the replay additionally invokes
Regina's complete guarded collapse tests, including the bigon and
pillow identification conditions.

After truncating a small cone at the ideal vertex, the finite--ideal
edge is a tame arc from the compact boundary to an ordinary interior
point, radial at its boundary end. A product tube capped at the
ordinary endpoint gives an explicit regular neighbourhood: a standard
PL \(4\)-ball meeting the boundary in a \(3\)-ball. Absorbing this
neighbourhood into the cusp is a boundary-collar change; together
with the guarded bigon and pillow flattenings, it gives a PL
homeomorphism between compact truncations. Extending its boundary
restriction conically gives a pointed PL homeomorphism of the finite
compactifications, distinct from the noninjective collapse quotient
map. The other guarded ideal moves have their checked local PL
interpretations at the cone point, so Lemma~\ref{lem:cone} applies
to the full path. The end-type and pinned-source audits listed in
Appendix~\ref{sec:hashes} record the complete checks.

\begin{proposition}\label{prop:exterior}
The exact compact exterior \(E\) constructed in
Section~\ref{sec:sphere} is PL homeomorphic to the compact BBH
exterior \(M\).
\end{proposition}

\begin{proof}
The compact moves identify \(E\) with \(E_{216}\). The cone
construction and the guarded ideal path identify its cone
compactification with that of a compact truncation of
\eqref{eq:Msig}. Lemma~\ref{lem:cone} gives the compact
identification. The checked endpoint isomorphism identifies that
truncation with \(M\).
\end{proof}

\section{An independent certificate for the punctured-torus fibration}
\label{sec:fiber}

This section concerns the BBH model \(M\) itself. It establishes the
bundle structure needed for the filling argument, independently of
the auxiliary triangulation \(K\). The normal hypersurface, its
cut manifold and the relative collapse are all specified in the
supplement.

\subsection{The normal cut}

A separate certificate starts at the ideal model \eqref{eq:Msig},
truncates it to 210 pentachora and uses 105 checked moves to
identify \(M\) with a compact 36-pentachoron triangulation.
In this triangulation, the specified normal hypersurface \(F\)
has 540 standard normal coordinates. It consists of 143
tetrahedron pieces and 120 prism pieces. The certificate establishes
that it is connected, two-sided and properly embedded, with one
\(2\)-sphere boundary component. Capping this boundary and applying
four checked \(3\)-dimensional moves gives Regina's \(3\)-torus
triangulation, with signature \sig{gvLQQedfedffrwawrhh}.
The cap is a PL \(3\)-ball, so \(F\) is a punctured \(3\)-torus.
Appendix~\ref{sec:hashes} identifies the normal vector and the
coordinate conventions recorded in the source certificate.

To implement the cut, normal pieces are realized as exact rational
hyperplane sections in each original pentachoron. Their complements
give 299 convex chunks. Their global quotient can identify distinct
boundary faces of a chunk. We therefore retain the local cell
occurrences and their explicit face-identification maps. The following
lemma interprets this recorded construction.

\begin{lemma}[Identification of the certified normal cut]
\label{lem:normalcut}
Let \(M'\) be the uncut flag triangulation obtained from the recorded
rational decompositions of the compact 36-pentachoron model \(M\),
and let \(F'\) be its marked slice subspace. The verified local face
lattices and facet maps determine a PL homeomorphism of pairs
\[
 (M',F')\congpl(M,F).
\]
Omitting the recorded slice-facet gluings therefore gives the PL cut
of \(M\) along \(F\), with its two boundary copies and the specified
inverse regluing.
\end{lemma}

\begin{proof}
In each original pentachoron the recorded planes give disjoint
standard normal tetrahedron and prism sections. Exact clipping
verifies that each new plane divides exactly one existing region,
and that every complementary region is a full-dimensional convex
polytope with the recorded complete face lattice. Across an original
paired facet, its vertex permutation and the ordered crossing ranks
give a carrier-preserving isomorphism of the whole subdivision face
poset. The checker verifies the inverse maps, their restrictions to
all original subfaces, and identity return maps around every
original face orbit. Passing to flags makes these face maps PL.

Straighten the local decompositions compatibly over the original
faces in increasing dimension, beginning with the ordered crossings
on edges. Extend each prescribed boundary map over the corresponding
standard normal pieces and then the complementary ball cells by
coning after a common subdivision. The restriction and return-map
identities ensure that these choices respect all original face
identifications. In dimension four this extends over the normal
\(3\)-balls and the convex complementary \(4\)-balls. The resulting
maps descend to a PL homeomorphism of pairs, preserving normal-piece
types, multiplicities and ordered matchings. The recorded output
gluing table is exactly the flag construction. Leaving its
slice-facet pairs unglued gives the marked cut, and the recorded
inverse regluing recovers \(M'\).
\end{proof}

The construction is deterministic in the reference replay
environment: the checker compares the numbered geometry data with
the archived reference, not only the final isomorphism signature.
It also performs the uncut reconstruction, verifies validity,
orientability and connectedness, and checks
\[
 f(M')=(2994,42608,139814,166536,66336),\qquad \chi(M')=0.
\]
The PL identification rests on Lemma~\ref{lem:normalcut} and its
incidence checks. Earlier exploratory group and boundary-recognition
calculations are not rerun by the nine-stage driver and are not used
to establish that identification.

Leaving precisely the hypersurface facets unglued gives the cut
manifold \(W\), together with its two actual copies \(F_-,F_+\) of
\(F\), its side \(C\), and its regluing map
\(\kappa:F_+\to F_-\). The counts are
\begin{align*}
 f(W)&=(4375,51844,155422,174288,66336),\\
 f(F_-)=f(F_+)&=(1381,9236,15608,7752).
\end{align*}
The cut is valid, orientable and connected, with \(\chi(W)=1\).
The two fiber copies are disjoint; each meets the side in its entire
boundary, a \(2\)-sphere with face vector \((106,312,208)\).
The side has 1392 tetrahedra and the whole boundary has 16,896.
The checker verifies all 331,680 output facet slots and all 7752
inverse regluing pairs, as well as the lower-dimensional seam data.

The side is independently certified to be \(S^2\times I\).
A checked 338-move reduction takes it to 39 tetrahedra. Capping its
two spherical boundary components gives a 67-tetrahedron closed
triangulation that reaches the standard two-tetrahedron \(S^3\)
in 54 checked moves. The PL \(3\)-dimensional Schoenflies theorem
then identifies the uncapped side as a cylinder, with its two actual
boundary components. Thus
\begin{equation}\label{eq:triad}
 \partial W=F_-\cup C\cup F_+,\qquad C\congpl S^2\times I.
\end{equation}
This is marked boundary data, not just a recognition of the closed
boundary as \(T^3\#T^3\).

\subsection{A product lemma with the required markings}

\begin{lemma}[Collapse to a boundary piece]\label{lem:unmarkedproduct}
Let \(W\) be a compact PL \(n\)-manifold and let \(A\) be a compact
codimension-zero PL submanifold of \(\partial W\), with collared
frontier. If a triangulation of \(W\) collapses to \(A\), then there
is a PL homeomorphism \(W\to A\times I\) taking every \(a\in A\)
to \((a,0)\).
\end{lemma}

\begin{proof}
This is the relative product theorem in
Stallings~\cite[Proposition~6.9.8, pp.~163--164]{Stallings1967},
whose conclusion fixes the
specified boundary piece pointwise. We recall its regular-neighbourhood
mechanism to make the boundary issue explicit. Take a regular
neighbourhood of \(A\) in \(W\), which is a product neighbourhood.
For the interior comparison, attach an external collar along all of
\(\partial W\), so the whole compact piece \(A\), including its
frontier \(\partial A\), lies in the interior of the enlarged
ambient manifold. Choose a smaller collar of \(A\) which the
comparison must fix. The collapse
and the collar collapses identify the enlarged product neighbourhood
and the enlarged \(W\) as regular neighbourhoods of the same
polyhedron. Relative regular-neighbourhood uniqueness identifies
them while fixing the smaller collar; restriction and collar
absorption give the stated homeomorphism. Stallings supplies the
relative construction in full. The corresponding regular-neighbourhood
and shelling results are Rourke--Sanderson 3.24--3.26, on
pp.~38--41~\cite{RS1972}.

The relative collar construction handles the frontier \(\partial A\)
and retains the marking of \(A\). In particular, it includes the
lateral boundary when \(\partial A\ne\varnothing\).
\end{proof}

\begin{lemma}[Marked product]\label{lem:markedproduct}
In Lemma~\ref{lem:unmarkedproduct}, suppose
\(\partial W=A\cup C\cup B\), where these are compact
codimension-zero PL boundary submanifolds with collared frontiers,
\(A\cap B=\varnothing\), \(A\cap C=\partial A\),
\(B\cap C=\partial B\), and there is a PL homeomorphism
\(c:\partial A\times I\to C\) with \(c(x,0)=x\) and
\(c(\partial A\times\{1\})=\partial B\). Then
\[
 (W;A,C,B)\congpl
 (A\times I;A\times0,\partial A\times I,A\times1),
\]
with the bottom fixed pointwise. No preassigned parametrization of
the interior of \(B\) is required.
\end{lemma}

\begin{proof}
Use Lemma~\ref{lem:unmarkedproduct} to obtain
\(h:W\to X=A\times I\), fixed on the bottom. Put
\[
 D=\cl\bigl(\partial X\setminus(A\times0)\bigr)
   =(\partial A\times I)\cup(A\times1).
\]
Both \(h(C)\) and \(\partial A\times I\) are closed collars of
\(\partial D=\partial A\times0\) in \(D\). Their complements,
\(h(B)\) and \(A\times1\), are PL manifolds with boundary, so
these collars extend past their inner frontiers. Ambient collar
uniqueness gives a PL isotopy of \(D\), fixed on \(\partial D\),
carrying the first collar onto the second
\cite[Chapter~4, Exercise~(4) following Theorem~4.24]{RS1972}.
Extend this isotopy by the identity on \(A\times0\). Its endpoint
extends over a boundary collar of \(X\): taper the isotopy from
its endpoint at the boundary to the identity at the inner edge
of the collar, and use the identity beyond it. This is a PL
homeomorphism of \(X\), fixed on the bottom
\cite[Proposition~3.22(ii)]{RS1972}. Composing with \(h\) aligns
the side. The complementary top is then aligned as well.
\end{proof}

The essential separating sphere in \(T^3\#T^3\) is also unique up to
isotopy~\cite[p.~1317]{HatcherWahl2005}, but the marking in
Lemma~\ref{lem:markedproduct} does not use that fact.

\subsection{The collapse and regluing certificates}

Two independently replayed logs give
\begin{equation}\label{eq:collapse}
  W\searrow F_-,\qquad W\searrow F_+,
\end{equation}
each in 209,144 elementary collapses. In each log the number of
pairs whose upper simplex has dimension \(4,3,2,1\) is respectively
66,336, 100,200, 39,614 and 2994. These counts account exactly for
\(f(W)-f(F_-)\), with \(f_4(F_-)=0\): a pair of upper dimension
\(d\) removes one \(d\)-simplex and one \((d-1)\)-simplex.
The replayer independently
reconstructs the face poset, checks its recorded checksum,
verifies the embedded target subcomplex and checks every free-face
pair. The upper face is maximal among live faces, the lower has
the required unique coface, and no face of the target is removed.
The final live complex is exactly the specified copy
of \(F\). These are actual elementary collapses, not merely
discrete Morse counts. Only one of the two logs is needed for the
product conclusion; the other supplies a second check with the
opposite preserved end.

The proper two-sided normal embedding supplies the local product
charts at the cut and its seams, hence the collared-frontier
hypotheses. Parametrize the certified side cylinder with its
recorded bottom; precompose the sphere coordinate by the inverse
of the bottom restriction to arrange \(c(x,0)=x\).
Apply Lemma~\ref{lem:markedproduct} to
\((W;F_-,C,F_+)\), using \eqref{eq:triad}. Let \(h_-\) and
\(h_+\) be the restrictions of the resulting product
homeomorphism to the bottom and top, with values in \(F\). Regluing
the recorded pair of fiber copies gives the mapping torus of
\begin{equation}\label{eq:return}
  \varphi=h_-\circ\kappa\circ h_+^{-1}:F\longrightarrow F.
\end{equation}

\begin{proposition}\label{prop:fibration}
The compact BBH exterior \(M\) is a PL bundle over \(S^1\) with
fiber a punctured \(3\)-torus, and the certified embedded
hypersurface \(F\) is a fiber.
\end{proposition}

The product homeomorphism and the full simplicial return map are
not themselves output by the collapse replayer. Their existence,
with the required end markings, is the conclusion of the two
explicit PL lemmas. This distinction is relevant to subsequent
attempts to turn the proof into a closed bistellar path.

\section{Cappell--Shaneson fillings and proof of the theorem}
\label{sec:CS}

\subsection{The integral monodromy}

The group calculation of BBH, Proposition~2.5, identifies the
kernel of the primitive map to \(\Z\) as \(\Z^3\), with action
\begin{equation}\label{eq:A}
 A=\begin{pmatrix}0&0&1\\1&0&0\\0&1&-1\end{pmatrix},\qquad
 \det A=1,\quad\det(A-I)=-1,\quad
 \det(tI-A)=t^3+t^2-1.
\end{equation}
In a commuting basis \(\alpha,\beta,\gamma\), the action is
\(\alpha\mapsto\beta\), \(\beta\mapsto\gamma\),
\(\gamma\mapsto\alpha\gamma^{-1}\).
BBH write the characteristic polynomial using \(\det(A-tI)\),
which has the opposite sign~\cite[Proposition~2.5]{BBH2012}.

The fibration in Proposition~\ref{prop:fibration} gives a
surjection \(\pi_1(M)\to\Z\). Since \(H_1(M;\Z)=\Z\), it is
the primitive homomorphism up to sign. Its fiber group is the
\emph{entire} kernel, so the integral action is \(A\), up to
integral change of basis and inversion of the base generator.
There is no inference from a rational characteristic polynomial
or a possibly finite-index sublattice. Choose the integral basis
and the orientation of the base circle so that this action is \(A\).

Cap the spherical boundary of the fiber by a \(3\)-ball and extend
\(\varphi\) over that ball by coning its boundary map. Denote the
extension by \(\bar\varphi:T^3\to T^3\). It has integral action
\(A\) and is orientation preserving, since the total bundle is
orientable. Let \(L_A\) be the linear torus automorphism induced by
\(A\). The map \(L_A^{-1}\bar\varphi\) acts trivially on
\(\pi_1(T^3)=\Z^3\), and is therefore homotopic to the identity
by asphericity. The closed \(T^3\) is Haken: it is irreducible
and contains an incompressible torus. Waldhausen's PL isotopy
theorem makes this composite isotopic to the identity, so
\(\bar\varphi\) is isotopic to \(L_A\)
\cite[Theorem~7.1, pp.~80--81, and Corollary~7.5, p.~85]{Waldhausen1968}.

The coning extension fixes the cone point and preserves the smaller
concentric cone neighbourhoods. Thus the capping balls form a PL
\(3\)-disk-bundle regular neighbourhood of their cone-point section;
deleting its interior recovers \(M\). The monodromy isotopy gives a fiber-preserving PL
homeomorphism to the linear torus mapping torus. There, any PL
section can be moved to the zero section by fiberwise translation.
In the model \((x,1)\sim(Ax,0)\), write the section as
\([s(t),t]\), so \(s(0)=A s(1)\). The map
\([x,t]\mapsto[x-s(t),t]\) is well-defined and takes it to
the zero section. Torus addition is PL, so this fiberwise
translation is a PL homeomorphism. Regular-neighbourhood
uniqueness then identifies
the complements. Thus \(M\) is, in the PL category, the
punctured-torus exterior in the Cappell--Shaneson construction
for \(A\). This argument does not assert uniqueness of a
parametrized section or of a relative-boundary monodromy.

\subsection{The filling and standardness}

The PL fillings of an \(S^2\times S^1\) boundary by
\(S^2\times D^2\) are represented by two attaching classes,
modulo maps extending over the filling. Gluck's homeotopy
calculation gives the factor reflections and the sphere twist
as generators~\cite[Theorem~5.1]{Gluck1962}. These have PL
representatives, and Hamstrom's three-dimensional comparison
theorem upgrades topological isotopies between PL maps to PL
isotopies~\cite[Section~2.1 and Theorem~3.1.2]{Hamstrom1974}:
apply injectivity on path components to the quotient of the two
maps in the topological identity component. The factor reflections
extend over \(S^2\times D^2\), leaving the identity and the
Gluck twist~\cite[Lemmas~15.1 and~15.3]{Gluck1962}; see also
Hatcher's smooth formulation~\cite{Hatcher1981}. Consequently
every filling of \(M\) by \(S^2\times D^2\), and in particular
the closed manifold \(N(S)\cup E\), is one of the two
Cappell--Shaneson spheres \(\Sigma_A^\varepsilon\),
\(\varepsilon\in\{0,1\}\).

For completeness, the standardness reduction can be checked
with small integer matrices. Let
\[
 B=A^{-1}=\begin{pmatrix}0&1&0\\1&0&1\\1&0&0\end{pmatrix},
 \qquad
 P=\begin{pmatrix}1&-1&1\\0&1&-1\\0&-1&2\end{pmatrix}.
\]
The inverse monodromy \(B\) has trace zero, as in the description
of this exterior in~\cite[Section~5]{BBS2026}.
Then \(\det P=1\), \(\det(B-I)=1\), and
\begin{equation}\label{eq:G}
 PBP^{-1}=G_{-2},\qquad
 G_m=\begin{pmatrix}0&1&0\\0&1&1\\1&0&m+1\end{pmatrix}.
\end{equation}
These identities are checked exactly in the supplement.
Inversion and integral conjugacy preserve the relevant
unordered pair of surgery spheres. Gompf's discussion at the
start of Section~3 and Example~3.1(a) show that both framings
for every \(G_m\), hence for \(G_{-2}\), give the standard
smooth \(S^4\)~\cite{Gompf2010}. A diffeomorphism between the
compatible smooth manifolds gives the required PL conclusion.

\begin{proof}[Proof of Theorem~\ref{thm:main}]
The 137-move certificate in Section~\ref{sec:closed} identifies
\(Q\) with \(K\). Section~\ref{sec:sphere}
decomposes \(K\) as \(N(S)\cup E\), with
\(N(S)\congpl S^2\times D^2\).
Proposition~\ref{prop:exterior} identifies the actual compact
exterior with \(M\), and Proposition~\ref{prop:fibration}
establishes its PL punctured-torus fibration. The integral
monodromy and capping argument identify its two fillings with
the Cappell--Shaneson pair for \(A\). Both are standard by
\eqref{eq:G} and Gompf's result. Therefore
\(Q\congpl K\congpl S^4\).
\end{proof}

\section{Verification and the trusted components}\label{sec:verification}

The computational supplement is archived with a version-specific DOI,
\url{https://doi.org/10.5281/zenodo.23479543}\cite{FaberMurphyCert2026}.
It contains the input triangulations, move sequences,
geometric constructions and independent replay programs. Its exact
identity and file concordance are given in Appendix~\ref{sec:hashes}.
The supplied archive is unchanged from the verified version.

Stages 2--9 of its driver supply the computational assertions in
Theorem~\ref{thm:main}. Stage 2 identifies \(K\) with \(Q\);
stages 3--8 identify the exterior of \(S\) with the BBH model;
stage 9 certifies the bundle and matrix data independently of \(K\).
Stage 1 verifies the additional application in
Appendix~\ref{sec:application}. It is not a premise of the main
theorem. The distributed driver runs all nine stages, retaining
their original order.

The reference environment is Python 3.12.14 with the Regina Python
distribution 7.4.1; the engine's \file{versionString()} reports 7.4.
The workers use POSIX resource limits. With Regina importable, run
the following from the extracted certificate directory, choosing a
fresh output directory outside the archive:
\begin{quote}\small
\verb|python REPLAY_ALL.py --output-directory /tmp/q-replay|
\end{quote}
For a separate Regina installation, add
\verb|--regina-path /absolute/path/containing/regina|.
Python's \file{-O} and \file{-OO} modes are prohibited because the
checkers use assertions. The driver executes all nine stages in
Table~\ref{tab:replay} and fails on a rejected move, a geometric
or incidence discrepancy, a checksum mismatch or an incorrect
endpoint. The nine worker stages in the archived reference run
took approximately 108 seconds, excluding the initial checksum and
source-consistency checks. Runtime is machine dependent.

\begin{table}[htbp]
\centering
\caption{The nine outer replay stages and their proof roles. The last
stage contains five internal checks.}
\label{tab:replay}
\begin{tabular}{@{}rp{3.1cm}p{9.0cm}@{}}
\toprule
& Stage & Certified conclusion\\
\midrule
1 & Additional application & The appendix's 42-vertex triangulation reaches \(Q\).\\
2 & Identification with \(Q\) & The specified \(K\) reaches exactly \(Q\).\\
3 & Local flatness & All five vertex-link pairs are unknotted.\\
4 & Exact geometry & \(N\) and \(E\) give the certified regular-neighbourhood decomposition.\\
5 & Compact moves & \(E_{2517}\) reaches \(E_{216}\) by 187 legal moves.\\
6 & Boundary cones & The 70-cone construction has exactly the claimed gluings.\\
7 & Ideal moves & The 445-move ideal path reaches the BBH model.\\
8 & End audit & The three collapse endpoint types and edge links are as stated.\\
9 & Fiber bundle data & The source identification, normal cut, both collapses, marked side and exact matrix identities pass.\\
\bottomrule
\end{tabular}
\end{table}

\paragraph{Finite certificates.}
The closed and exterior move logs specify face indices in the
current triangulation; they are replayed in order, with no hidden
intermediate canonical relabelling. The coning and geometric
checkers reconstruct their objects independently. The normal-cut
audit checks carrier-preserving face identifications and regluing,
and the collapse replayer reconstructs all face incidences and
checks every elementary collapse. The latter checks 5,466,332
local face occurrences when validating the embedded complex and
its preserved target subcomplexes.

\paragraph{Shared software kernel.}
Independence of checker implementations is not independence of
all mathematical software. Generalized move legality, exact
\(3\)-manifold recognition, the BBH source's ideal truncation,
normal-hypersurface materialization and the \(3\)-dimensional
spherical capping routines use Regina
7.4.1~\cite{Regina}. The guarded moves enforce sufficient,
move-specific combinatorial conditions for PL preservation:
identification and bigon/pillow tests for edge collapses, prescribed
degree-two face configurations and identification tests for \(2\)--\(0\) moves,
boundary-incidence tests for shellings, and the prescribed link
configuration for \(4\)--\(4\) moves. The audited source revision is
recorded in Appendix~\ref{sec:hashes}.
In particular, the finite--ideal collapse relies on the full
guarded implementation in
\file{engine/triangulation/dim4/simplify.cpp}; its independent
endpoint-link audit supplements that implementation. The small
link and knot-exterior recognition checks invoke Regina afresh
and do not claim separate logged certificates for every internal
recognition operation. The geometric arithmetic is exact rational
arithmetic. This is a replayable computer-assisted proof, not a
formal proof in a proof assistant.

\paragraph{Theorems, not machine assertions.}
The following implications are supplied by mathematics:
local flatness and zero normal Euler number give the product
normal neighbourhood (Lemma~\ref{lem:normal}); the pointed finite
cone equivalence gives a compact exterior equivalence
(Lemma~\ref{lem:cone}); the verified local chunk data identify the
normal cut (Lemma~\ref{lem:normalcut}); the collapse and marked side give the
product triad (Lemmas~\ref{lem:unmarkedproduct}
and~\ref{lem:markedproduct}); \(3\)-torus rigidity identifies the
closed monodromy; the boundary mapping-class calculation leaves
two filling classes; and Gompf establishes standardness for both.
The checkers do not replace any of these theorems.

\paragraph{Computer assistance.}
ChatGPT (OpenAI) and Claude Fable (Anthropic) assisted with
mathematical exploration, development and review of computational
code, reference searches, and manuscript drafting and revision. The
computational assertions are supported by the archived inputs and
deterministic replay programs; the PL interpretations are stated as
lemmas and supported by the cited sources.

\section{Further constructive questions}\label{sec:open}

An explicit closed Pachner path from \(Q\) to a standard sphere
triangulation remains to be constructed. Burke, Burton and Spreer
prove that any such path must reach at least twelve
pentachora~\cite[Proposition~4]{BBS2026}. PL standardness guarantees
bistellar equivalence after passing to genuine simplicial
triangulations~\cite{Pachner1991,Lickorish1999}. The present proof
uses a geometric decomposition and PL identifications, and supplies
no useful upper bound on the length or height of a closed path.

The exterior certificate passes from \(E_{216}\) through its coned
ideal presentation. No direct compact move sequence between
\(E_{216}\) and the 36-pentachoron model of \(M\) was obtained;
the ideal passage supplied the successful identification. Finding
a compact sequence could help turn the decomposition into a closed
Pachner path.

The choice of surgery framing could be determined after fixing a
boundary identification with the Cappell--Shaneson model. Both
choices yield the standard ambient sphere, but this additional
marking would specify the surgery description. Similarly, the
product homeomorphism supplied by the relative collapse could be
converted into an explicit simplicial product map and a recorded
return map. These constructions would connect the finite
certificates more directly to geometric standardness proofs.

The marked \(2\)-knot \(S\subset K\congpl S^4\) is nontrivial:
its exterior group \(\Z^3\rtimes_A\Z\) is noncyclic, whereas
the standard \(2\)-unknot has exterior group \(\Z\).
Its relation to the particular cut-open pieces exhibited in the
census paper remains to be made explicit. The pair \((K,S)\)
provides concrete geometric data with which to investigate that
relationship and to develop simplification moves that preserve
the knot exterior.

\section*{Acknowledgements}
The use of ChatGPT (OpenAI) and Claude Fable (Anthropic) for this work
is detailed in Section~\ref{sec:verification}.

\clearpage
\appendix
\small
\section{Reproducibility data and archive concordance}\label{sec:hashes}

The archive retains the filenames used when its certificates were
generated. In particular, the names \file{K347}, \file{current347}
and \file{floor2}, and the label \file{g2=347}, refer to the auxiliary
triangulation \(K\) of this paper. The last label records
\(g_2(K)=f_1(K)-5f_0(K)+15=347\); this statistic plays no role in
the proof. These names introduce no additional mathematical objects
or assumptions.

\paragraph{The auxiliary triangulation and sphere.}
The facet list for \(K\) is
\file{closed_anchor/floor2/floor2_source_facets.json};
its 137-move certificate is
\file{closed_anchor/floor2/floor2_Q_certificate.json}.
The geometric construction uses the matching data in
\file{geometry/source_checkpoint.json} and
\file{geometry/source_facets.json}.
The five marked vertices have the following identifiers:
\[
\begin{array}{c|rrrrr}
\text{paper notation} & u & v & a & b & c\\ \hline
\text{facet-file label} & 0 & 3133 & 964 & 1628 & 1711
\end{array}
\]
These labels are arbitrary vertex identifiers. The canonical
facet-set SHA-256 for \(K\) is
\begin{quote}\small
\nolinkurl{4138e1e62b27d003f0949c7f8a49c1a10de91b5781394ad988d32ea06471d39d}.
\end{quote}

To compute a canonical facet-set checksum, sort the vertex labels
within each facet, sort the resulting facet rows lexicographically,
and encode the array as compact JSON with comma and colon
separators, in UTF-8 without a final newline. Hash those bytes.
This is a checksum of the facet array; a containing JSON file may
also carry metadata and have a different byte checksum.

\paragraph{The compact model and the normal vector.}
The compact source and move data are in
\file{bbh_fiber_audit/source_anchor/bbh_small_model.json}.
The vector of \(F\) is the field \file{normal_vector_standard} in
\file{bbh_fiber_audit/source_anchor/candidate26_recognition.json},
where \file{normal_index} is 26.
The source identification records a
canonical reload after move 55 and the final canonical coordinate
system used for the 540-entry normal vector.

\paragraph{The fixed supplement.}
The archive supplied with the manuscript is
\file{core42_standardness_certificate_20261019.zip}.
It contains 123 files, has size 5,603,792 bytes, and has SHA-256
\begin{quote}\small
\nolinkurl{acb0d895605b0cdd15b99a8df52cb47a8c9543bc1b64d15e1cbeed20137db3f2}.
\end{quote}
The manifest checks every payload. The top-level \file{REPORT.txt}
records the completed chain; earlier working reports remain as
historical records. The archive's filename reflects the additional
application below. Its first replay stage and final comparison of
closed endpoints remain part of the distributed driver.
The public artifact release is archived at
\url{https://doi.org/10.5281/zenodo.23479543}.
The corresponding GitHub repository is
\url{https://github.com/ThePigLA/four-pentachoron-sphere-Q}.

\paragraph{The audited Regina source.}
The revision used for the source audit is
\begin{quote}\small
\nolinkurl{219d658e5e51cc003807fc871fb3534a448a972f}.
\end{quote}
The edge-collapse implementation is
\file{engine/triangulation/dim4/simplify.cpp}.
The supplement's end-type and source audits are
\file{theory/IDEAL_COLLAPSE_END_TYPES.json} and
\file{theory/REGINA_PINNED_SOURCE.json}.

\section{Application to an additional simplicial triangulation}
\label{sec:application}

The supplement also includes a 42-vertex simplicial triangulation,
labelled \(\core\), with face vector
\((42,537,1748,2085,834)\). This is an additional application of
Theorem~\ref{thm:main}; it is not used in its proof.

\begin{corollary}
The specified triangulation \(\core\) is PL homeomorphic to \(Q\)
and hence to \(S^4\).
\end{corollary}

\begin{proof}
Its separately replayed 78-move certificate reaches exactly
\eqref{eq:Qsig}. It consists of 40 edge collapses, one
\(4\)--\(2\) move, six triangle \(2\)--\(0\) moves, five edge
\(2\)--\(0\) moves and 26 \(3\)--\(3\) moves. Each move is checked
for legality. Apply Theorem~\ref{thm:main}.
\end{proof}

The raw input is \file{closed_anchor/original_core42_input.json};
the normalized input and certificate are
\file{closed_anchor/original_core42_canonical_facets.json} and
\file{closed_anchor/original_core42_reduction_certificate.json}.
The canonical facet-set SHA-256 is
\begin{quote}\small
\nolinkurl{e010a76015367f6b5558c1aadde3a28262fcb00c3560bb6609305f24a714f80e}.
\end{quote}
The raw input-file SHA-256 is
\begin{quote}\small
\nolinkurl{7a37bf0aa2c334b9830681781e3a16c2506c8db3dc1c3fd31aa0b43d2a3af96f}.
\end{quote}

{\footnotesize
\bibliographystyle{plain}
\bibliography{references}
}

\end{document}